\documentclass[11pt,a4paper]{amsart}

\usepackage{amssymb}
\usepackage{amsfonts}
\usepackage{amsmath, amsthm}
\usepackage[all,cmtip]{xy}

\usepackage{multirow}
\usepackage{pgf}
\usepackage{yfonts}
\usepackage[colorlinks=true,linkcolor=blue,citecolor=red]{hyperref}

\usepackage{marginnote}

\theoremstyle{plain}
\newtheorem{theorem}{Theorem}[subsection]

\newtheorem{proposition}[theorem]{Proposition}
\newtheorem{corollary}[theorem]{Corollary}
\newtheorem{definition}[theorem]{Definition}
\newtheorem{remark}[theorem]{Remark}

\newcommand{\N}{\mathbb{N}}
\newcommand{\Z}{\mathbb{Z}}
\newcommand{\Q}{\mathbb{Q}}
\newcommand{\R}{\mathbb{R}}
\newcommand{\C}{\mathbb{C}}
\newcommand{\K}{\mathbb{K}}
\newcommand{\Pj}{\mathbb{P}}
\newcommand{\sfM}{\mathsf{M}}

\DeclareMathOperator{\GL}{{\mathsf{GL}}}
\DeclareMathOperator{\SL}{{\mathsf{SL}}}
\DeclareMathOperator{\Proj}{Proj}

\newcommand{\bond}{\mathfrak{b}}
\newcommand{\lineL}{\mathcal{L}}
\newcommand{\calC}{\mathcal{C}}

\newcommand{\pzero}{\mathsf{0}}
\newcommand{\pone}{\mathsf{1}}
\newcommand{\pfs}{\dotplus}
\DeclareMathOperator{\LS}{{\mathsf{LS}}}

\DeclareMathOperator{\lcm}{lcm}
\DeclareMathOperator{\cleq}{\preccurlyeq}
\DeclareMathOperator{\rlex}{rlex}
\DeclareMathOperator{\leqt}{\leq^t}

\DeclareMathOperator{\leqtrlex}{\leq^t_{\rlex}}

\DeclareMathOperator{\supp}{supp}
\DeclareMathOperator{\Hom}{Hom}
\DeclareMathOperator{\Tor}{Tor}

\newcommand{\myref}[2]{\hyperref[#2]{#1 \ref{#2}}}

\title{On Some Properties of LS Algebras}

\author{Rocco Chiriv\`\i}
\address{Dipartimento di Matematica e Fisica ``Ennio De Giorgi'', Universit\`a del Salento, Lecce, Italy}
\email{rocco.chirivi@unisalento.it}



\begin{document}

\begin{abstract}
The discrete LS algebra over a totally ordered set is the homogeneous coordinate ring of an irreducible projective (normal) toric variety. We prove that this algebra is the ring of invariants of a finite abelian group containing no pseudo-reflection acting on a polynomial ring. This is used to study the Gorenstein property for LS algebras. Further we show that any LS algebra is Koszul.
\end{abstract}

\maketitle

\noindent Accepted by \emph{Communications in Contemporary Mathematics},\\
{\tt DOI:10.1142/S0219199718500852}

\subsection{Introduction} LS Algebras have been introduced in \cite{chirivi_ls} as an algebraic framework for standard monomial theory; they are a generalization of Hodge algebras as defined in \cite{dep}.

Standard monomial theory started with Hodge’s study of Grassmannians in \cite{hodge} and (with Pedoe) in \cite{hodge_pedoe}. A similar result for the space of matrices was found by Doubilet, Rota and Stein in \cite{DRS}; then this was generalized to symmetric and antisymmetric matrices by De Concini and Procesi in \cite{DP}.

A systematic program for the development of a standard monomial theory for quotients of reductive groups by parabolic subgroups was started by Seshadri in \cite{seshadri} where the case of minuscule parabolics is considered. Further, in \cite{seshadri_lakshmibai} Seshadri and Lakshmibai noted that the above mentioned results was specializations of their general theory.

This program was finally completed by Littelmann. Indeed, in \cite{littelmann_paths}, he found a combinatorial character formula for representations of symmetrizable Kac-Moody groups introducing the language of LS-paths. Moreover, in \cite{littelmann_contracting} he constructed a basis of the irreducible representations associated to LS-paths and proved that this basis defines a standard monomial theory for Schubert varieties of symmetrizable Kac-Moody groups. A more precise version of certain defining relations was proved by Littelmann in collaboration with Lakshmibai and Magyar in \cite{LLM}. At the same time the notion of LS algebras was being introduced by the author and these refined relations exactly matched with the order requirement for LS algebras.

An LS algebra is modelled on the combinatorial structure of LS-paths. Given a partially ordered set $(S,\leq)$ with certain kind of multiplicities $\bond$ called bonds verifying some mild compatibility conditions, one can define the abstract notion of an LS-path over $(S,\leq,\bond)$. An LS algebra is a graded algebra having the set of LS-paths as a vector space basis and such that the product of two LS-paths is a linear combination, called straightening relation, of LS-paths verifying certain order condition.

The above cited Littelmann's results in \cite{littelmann_contracting} and \cite{LLM} can be expressed in terms of LS algebras as follows. Let $G$ be a semisimple algebraic group, or a symmetrizable Kac-Moody group as well, let $B$ be a Borel subgroup of $G$ and $P$ a parabolic subgroup containing $B$. Denote by $X$ a Schubert variety in the  partial flag variety $G/P$ and let $\lineL$ be an ample line bundle on $G/P$. The ring $R=\oplus_{n\geq 0}H^0(X,\lineL^{\otimes n})$ is the coordinate ring of the cone over an embedding of $X$. As proved in \cite{chirivi_ls}, $R$ is an LS algebra over the poset with bonds associated with the pair $P,\lineL$. This example is the actual \emph{raison d'\^etre} of LS algebras; the case of grassmannians is covered by Hodge algebras and was the other source of inspiration.

In order to treat the multicone over a partial flag variety, a multigraded version of LS algebras has been developed in \cite{chirivi_multicone}. Further steps have been the application of this circle of ideas to symmetric varieties in \cite{equations}, to model varieties in \cite{CM_modelVarieties} and next to wonderful varieties in \cite{wonderful}. An explicit computation of the straightening relations for the spin module of orthogonal groups is presented in \cite{pfaffians}; this completes the explicit standard monomial theory for minuscule cases of classical groups.

The aim of the present paper is to continue the study of general properties of LS algebras begun in \cite{chirivi_ls}.

In particular we see in details why the simplest LS algebra over a totally ordered set with bonds $(S,\leq,\bond)$, called discrete LS algebra, is the homogeneous coordinate ring of an irreducible projective (normal) toric variety. Moreover the discrete algebra for a general poset is the homogeneous coordinate ring of a glueing of irreducible (normal) toric varieties along (normal) toric subvarieties.

Next we show that, for a totally ordered set, the discrete algebra is also the ring of invariants of a finite abelian group containing no pseudo-reflection acting on a polynomial ring. This allows a very neat criterion for Gorensteiness of LS algebras over totally ordered sets.

Finally, by translating a standard monomial theory in the language of Gr\"obner basis, we study the Koszul property for LS algebras over general poset with bonds. It turns out that any LS algebra is Koszul.

\subsection{Poset with bonds} Let $(S,\leq)$ be a finite partially ordered set, a \emph{poset} for short, with a unique minimal element $\pzero$ and a unique maximal element $\pone$; assume further that the poset is graded, i.e. all complete chains with fixed initial and final elements have the same number of elements. 

The element $\tau$ \emph{covers} the element $\sigma$ in $S$ if (1) $\sigma<\tau$ and (2) if $\tau\leq\eta\leq\sigma$ then either $\eta=\tau$ or $\eta=\sigma$. In what follows we always denote by $N$ the \emph{length} of $S$, i.e. the number os covering relations in a (hence any) complete chain from $\pzero$ to $\pone$.
\begin{definition}\label{definition_bond}
A \emph{set of bonds} on $S$ is a map $\bond$ from pairs $(\sigma,\tau)$, with $\tau$ covering $\sigma$, to positive integers such that: given two complete chains
\[
\sigma = \eta_{i,1}<\eta_{i,2}<\cdots<\eta_{i,r} = \tau,\quad i = 1,2
\]
from $\sigma$ to $\tau$ in $S$ we have
\[
\gcd_{1\leq i\leq r-1}\bond(\eta_{1,i},\eta_{1,i+1}) \,\,= \gcd_{1\leq i\leq r-1}\bond(\eta_{2,i},\eta_{2,i+1}).
\]
\end{definition}
The map $\bond$ can be extended to comparable pairs $(\sigma,\tau)$ as the greatest common divisor on a complete chain from $\sigma$ to $\tau$. For an element $\sigma\in S$, we denote by $M_\sigma$ the least common multiple of all bonds $\bond(\eta,\tau)$, with $\tau$ covering $\eta$, and either $\eta=\sigma$ or $\tau=\sigma$. Note that, by the way we have defined the bonds on comparable pairs, $M_\sigma$ is also the least common multiple of all bonds $\bond(\eta,\tau)$, with $\eta<\tau$ and either $\eta=\sigma$ or $\tau=\sigma$.

The order complex $\Delta(S)$ of the poset $(S,\leq)$ is the abstract simplicial complex having as simplexes all the chains of $S$. The facets of $\Delta(S)$ are the maximal chains of $S$ and, being $S$ graded, $\Delta(S)$ is pure dimensional (see, for example, \cite{bjorner}.)

\subsection{LS-paths} Let $(S,\leq,\bond)$ be a poset with bonds.
\begin{definition}\label{definition_LSpath}
A function $\pi:S\longrightarrow\Q$ is an \emph{LS-path} of \emph{degree} $\deg\pi = r\in\N$ if the following conditions hold:
\begin{itemize}
\item[(1)] $\pi(\sigma)\geq0$ for any $\sigma\in S$,
\item[(2)] $\supp\pi$, defined as the set of $\sigma\in S$ such that $\pi(\sigma) \neq 0$, is a totally ordered subset of $S$,
\item[(3)] if $\supp\pi=\{\sigma_1<\sigma_2<\cdots<\sigma_n\}$ then, for all $j=1,\ldots,n-1$,
\[
\bond(\sigma_j,\sigma_{j+1})\sum_{i=1}^j\pi(\sigma_i)\in\N
\] and
\[
\sum_{i=1}^n\pi(\sigma_i)=r.
\]
\end{itemize}
\end{definition}
It follows at once from the definition that $M_\sigma\pi(\sigma)$ is an integer for any LS-path $\pi$ and any $\sigma\in S$. Note that in (3) we may equivalently suppose that $\supp\pi\subseteq\{\sigma_1<\sigma_2<\cdots<\sigma_n\}$. Further the condition (3) implies that $\pi(\sigma)\in\Q$ for any $\sigma\in S$; so the conditions (1), (2) and (3) characterize the LS-paths of degree $r$ in the set of real valued functions on $S$.

We denote by $\LS_r$ the set of all LS-paths of degree $r$ and by $\LS$ the set of all LS-paths. 
Given a chain $\calC$ in $S$, we denote by $\LS(\calC)$ the set of LS-paths having support contained in $\calC$, accordingly $\LS_r(\calC)$ is the set of LS-paths of degree $r$ in $\LS(\calC)$.

The operation of addition of functions $S\longrightarrow\Q$ is denoted by $\pfs$, in order to better distinguish it from addition in other rings. We say that the LS-paths $\pi_1,\pi_2,\ldots,\pi_r$ have \emph{comparable supports} if there exists a chain in $S$ containing $\supp\pi_1\cup\cdots\cup\supp\pi_r$; i.e., there exists a chain $\calC$ such that $\pi_1,\pi_2,\ldots,\pi_r\in\LS(\calC)$. In such a case the sum $\pi_1\pfs\pi_2\pfs\cdots\pfs\pi_r$ is an LS-path of degree $\deg\pi_1 + \cdots + \deg\pi_r$. More is true since
\begin{proposition}[Proposition 3 in \cite{chirivi_ls}]
If $\pi$ is an LS-path of degree $r$ then there exist LS-paths $\pi_1,\pi_2,\cdots,\pi_r$ of degree $1$, such that: $\max\supp\pi_h\leq\min\supp\pi_{h+1}$, for all $h=1,2,\ldots,r-1$, and $\pi=\pi_1 \pfs \pi_2 \pfs \cdots \pfs \pi_r$ as functions on $S$.
\end{proposition}

A (formal) monomial $\pi_1\pi_2\cdots\pi_r$ of LS-paths is \emph{standard} if $\max\supp\pi_h\leq\min\supp\pi_{h+1}$ for $h=1,2,\ldots,r-1$. If the LS-paths $\pi_1,\pi_2,\ldots,\pi_r$ have comparable supports then $\pi=\pi_1\pfs\cdots\pfs\pi_r$ is an LS-path and, by the previous proposition, there exist LS-paths $\pi_{0,1},\pi_{0,2},\dots,\pi_{0,r}$ of degree $1$ such that (1) $\pi=\pi_{0,1}\pfs\cdots\pfs\pi_{0,r}$ and (2) the monomial $\pi_{0,1}\cdots\pi_{0,r}$, called the \emph{canonical form} of $\pi_1\pi_2\cdots\pi_r$, is standard. We say also that $\pi_{0,1} \pfs \pi_{0,2} \pfs \cdots \pfs \pi_{0,r}$ is the canonical form of $\pi$.

The definition of LS algebra in next Section uses certain orders that we introduce here. Let $\leqt$ be a total order on $S$ refining the partial order $\leq$. The set of all functions $\pi:S\longrightarrow\Q$ can be totally ordered using the reverse lexicographic order: we define
\[
\pi\leqtrlex\pi'
\]
if and only if either $\pi=\pi'$ or, denoting by $\sigma$ the $\leqt$--maximal element with $\pi(\sigma)\neq\pi'(\sigma)$, we have $\pi(\sigma)<\pi'(\sigma)$. We define also $\pi\trianglelefteq\pi'$ if $\pi\leqtrlex\pi'$ for any total order $\leqt$ refining $\leq$ on $S$; the relation $\trianglelefteq$ is a partial order.

\subsection{LS algebras}\label{subsection_LSAlgebra} Let $\K$ be a field, let $A$ be a commutative $\K$--algebra, fix an injection $j:\LS_1\hookrightarrow A$ and extend it to $\LS$ using the canonical form for higher degree LS-paths: if $\pi=\pi_1 \pfs \pi_2 \pfs \cdots \pfs \pi_r$ is the canonical form of $\pi$, then define $j(\pi)=j(\pi_1)j(\pi_2)\cdots j(\pi_r)$. Usually the map $j$ will be omitted and we simply write $\pi\in A$ instead of $j(\pi)\in A$.
\begin{definition}\label{definition_LSalgebra}
The algebra $A$ is an \emph{LS algebra} over the poset with bonds $(S,\leq,\bond)$ if
\begin{enumerate}
\item[(LSA1)] the set of LS-paths is a basis of $A$ as a $\K$--vector space and the LS-path degree induces a grading for $A$,
\item[(LSA2)] if $\pi_1\pi_2$ is a non-standard monomial of degree $2$ and
\[
\pi_1\pi_2=\sum_j c_j\pi_{j,1}\pi_{j,2}
\]
is the unique relation, called a \emph{straightening relation}, expressing $\pi_1\pi_2$ as a $\K$--linear combination of standard monomials, as guaranteed by \textnormal{(LSA1)}, then $\pi_1 \pfs \pi_2\trianglelefteq\pi_{j,1} \pfs \pi_{j,2}$, for any $j$ such that $c_j\neq0$,
\item[(LSA3)] if $\pi_1\pi_2$ is a non-standard monomial of degree $2$ and $\pi_1$,$\pi_2$ have comparable supports then the canonical form of $\pi_1 \pfs \pi_2$ appears with coefficient $1$ in the straightening relation for $\pi_1\pi_2$.
\end{enumerate}
\end{definition}
In the sequel, if not otherwise stated, we always write a straightening relation implicitly assuming that $c_j\neq0$.
\begin{remark}\label{remark_straighteningRelations}
As proved in \cite{chirivi_ls}, the straightening relations in \textnormal{(LSA2)} generate all the relations among the generators in $\LS_1$ of $A$; in other words, by using \textnormal{(LSA2)} a suitable number of times, it is possible to write any non-standard monomial $\pi_1\pi_2\cdots\pi_r$, with $\pi_1,\ldots,\pi_r$ of degree $1$, as a linear combination of LS-paths of degree $r$, equivalently, of standard monomials of degree $r$.
\end{remark}
\begin{remark}\label{remark_minimal}
If $\pi_1,\pi_2$ have comparable supports, as in \textnormal{(LSA3)} above, and $\pi_{0,1},\pi_{0,2}$ is the canonical form of the non-standard monomial $\pi_1\pi_2$, then $\pi_{0,1} \pfs \pi_{0,2}=\pi_1 \pfs \pi_2$, hence $\pi_{0,1}\pi_{0,2}$ is the $\trianglelefteq$--minimal monomial appearing in the straightening relation for $\pi_1\pi_2$.
\end{remark}
The above definition may be slightly generalized allowing the coefficient of $\pi_{0,1}\pi_{0,2}$ to be any non-zero scalar. (Compare with the definition of special LS algebra in \cite{chirivi_ls}.)

We call the elements $\pi\in A$, for $\pi\in\LS_1$, a \emph{path basis} for the LS algebra $A$. Note that an algebra may have various path basis.

If all bonds equal $1$ then we say that $A$ is an Hodge algebra (see \cite{dep}).
\begin{definition}\label{definition_discrete}
The \emph{discrete} LS algebra $A$ over the poset with bonds $(S,\leq,\bond)$ has the simplest straightening relations, i.e. for any non-standard monomial $\pi_1\pi_2$ of degree $2$ we have: either $\pi_1\pi_2=\pi_{0,1}\pi_{0,2}$, if $\pi_1,\pi_2$ have comparable support and $\pi_{0,1}\pi_{0,2}$ is the canonical form of $\pi_1\pi_2$, or $\pi_1\pi_2=0$ if the supports of $\pi_1$ and $\pi_2$ are not comparable.
\end{definition}

We see a concrete realization of the discrete algebra $A$ over a totally ordered set $S=\{\sigma_0<\sigma_1<\cdots<\sigma_N\}$ with bonds $\bond$. First of all any pair of LS-paths has comparable supports being $S$ totally ordered. Now let $\K[S]$ be the polynomial algebra with variables $x_\sigma$, $\sigma\in S$. Given an LS-path $\pi$, let $x^\pi$ be the monomial $\prod_{\sigma}x_\sigma^{M_\sigma\pi(\sigma)}$. These monomials are linearly independent as $\pi$ runs in $\LS$ and $x^{\pi \pfs \pi'}=x^{\pi}x^{\pi'}$. In particular the relations $x^{\pi}{x^{\pi'}} = x^{\pi_0}x^{\pi_0'}$ holds if $\pi_0\pi_0'$ is the canonical form of $\pi\pi'$. Hence we have an injective morphism of algebras by linearly extending $A\ni\pi\longmapsto x^\pi\in\K[S]$. Clearly the image of this map is isomorphic to $A$; in the sequel we will identify these two algebras.
\begin{remark}\label{remark_discreteIsDomain}
Being a subring of a polynomial ring, a discrete LS algebra over a totally ordered set is a graded domain. Moreover, as proved in \cite{chirivi_ls}, any LS algebra over a totally ordered set is a domain.
\end{remark}

\smallskip

Recall the following crucial degeneration result (see \cite{chirivi_ls}).
\begin{theorem}\label{theorem_degeneration}
Let $A$ be an LS algebra over the poset with bonds $(S,\leq,\bond)$. Then there exists a flat $\K[t]$--algebra $\mathcal{A}$ such that: $\mathcal{A}/(t-a)$ is isomorphic to $A$ for all $a\in\K^*$ and $\mathcal{A}/(t)$ is isomorphic to the discrete LS algebra over $(S,\leq,\bond)$.
\end{theorem}
This allows to recover certain properties of an LS algebra by studying the same properties for the discrete LS algebra, a completely combinatorial object.

Sometimes one can go a step further and consider the Stanley-Reisner ring $\K\{S\}$ of $(S,\leq)$ that we can shortly define as the discrete algebra over the poset with bonds $(S,\leq,1)$. This happens for example for the Cohen-Macaulay property (see \cite{chirivi_ls}).
\begin{theorem}
An LS algebra over the poset with bonds $(S,\leq,\bond)$ is Cohen-Macaulay if and only if $\K\{S\}$ is so.
\end{theorem}
This is particularly interesting since $\K\{S\}$ is Cohen-Macaulay if and only if the order complex $\Delta(S)$ is Cohen-Macaulay (see \cite{reisner} or \cite{chirivi_ls} for the definition of a Cohen-Macaulay complex).

\subsection{Discrete algebra and toric varieties}

The following results are already stated in \cite{chirivi_ls} with somehow implicit proofs.
\begin{proposition}\label{proposition_discreteTotallyOrderedToric}
The discrete LS algebra over a totally ordered set with bonds is a normal domain. In particular it is the homogeneous coordinate ring of an irreducible projective $N$--dimensional (normal) toric variety.
\end{proposition}
\begin{proof}
The discrete algebra $A$ over the poset with bonds $(S,\leq,\bond)$ is a domain by Remark \ref{remark_discreteIsDomain}; we prove that it is normal by showing that it is the coordinate ring of an affine toric variety in the sense of \cite{fulton}. Let $\sfM$ be the abelian group of all functions $\gamma:S\longrightarrow\Z$ such that
\[
\begin{array}{l}
\bond(\sigma_j,\sigma_{j+1})\displaystyle\sum_{i=0}^j\frac{\gamma(\sigma_i)}{M_{\sigma_i}}\in\Z,\quad\textrm{for all }j=0,\ldots,N-1,\\[1em]
\displaystyle\sum_{i=0}^N\frac{\gamma(\sigma_i)}{M_{\sigma_i}}\in\Z.\\
\end{array}
\]
This is clearly a lattice in $\R^S$ and, denoting by $\Gamma$ the cone $\R_{\geq 0}^S$, the group algebra associated to the semigroup $\sfM\cap\Gamma$ is isomorphic $A$. So $A$ is the coordinate ring of the irreducible (normal) toric variety associated to the torus $\Hom(\sfM,\Z)$ and to the cone dual to $\Gamma$.

Moreover $\Proj(A)$ is projectively normal hence it is a projective (normal) toric variety; its dimension is $N$ as proved in \cite{chirivi_ls}.
\end{proof}

We consider now a general poset with bonds $(S,\leq,\bond)$; we want to prove that the discrete LS algebra $A$ over $(S,\leq,\bond)$ is the homogeneous coordinate ring of the glueing of (normal) toric varieties along (normal) toric subvarieties.

Let $R = \K[x_\pi\,|\,\pi\in\LS_1]$ be the polynomial ring with an indeterminate $x_\pi$ for each LS-path of degree $1$ and denote by $I$ the ideal generated by
\[
\mathcal{R}_{\pi_1,\pi_2} = \left\{
\begin{array}{ll}
x_{\pi_1}x_{\pi_2}-x_{\pi_{0,1}}x_{\pi_{0,2}} & \textrm{if }\pi_1,\pi_2\textrm{ have comparable supports, where}\\
 & \pi_{0,1}\pi_{0,2}\textrm{ is the canonical form of }\pi_1\pi_2\\
 & \\
x_{\pi_1}x_{\pi_2} & \textrm{otherwise}
\end{array}
\right.
\]
as $\pi_1\pi_2$ runs over all non-standard monomials of degree $2$. It is clear by Definition \ref{definition_LSalgebra} and Remark \ref{remark_straighteningRelations} that $A\simeq R / I$.

Now let $\calC$ be a chain in $S$, then $(\mathcal{C},\leq_{|\mathcal{C}},\bond_{|\mathcal{C}})$ is a totally ordered set with bonds. Consider the ideal $I_{\mathcal{C}}$ of $R$ generated by $I$ and $x_\pi$ for all $\pi\in\LS_1$ such that $\supp\pi\not\subseteq\mathcal{C}$.
Again by Definition \ref{definition_LSalgebra} and Remark \ref{remark_straighteningRelations} the quotient $R / I_\mathcal{C}$ is the discrete LS algebra $A_\mathcal{C}$ of the poset with bonds $(\mathcal{C},\leq_{|\mathcal{C}},\bond_{|\mathcal{C}})$. Moreover, by Proposition~27 in \cite{chirivi_ls}, $A$ and $A_{\mathcal{C}}$ are reduced.

Now let $r=|\LS_1|$ and consider the varieties $X=\Proj(A)$ and $X_\calC=\Proj(A_\calC)$ for $\calC$ a chain in $S$; we have the inclusions $X_\calC\,\subseteq\,X\,\subseteq\,\Pj^r$. Since $I_{\calC_1}+I_{\calC_2} = I_{\calC_1\cap\calC_2}$ for any pair of chains $\calC_1,\calC_2$ of $S$, we find $X_{\calC_1}\cap X_{\calC_2}=X_{\calC_1\cap\calC_2}$ as schemes.

Putting together the quotient maps we find a map
\[
A\,\longrightarrow\,\prod_{\mathcal{C}}A_{\mathcal{C}}
\]
where the product runs over all maximal chains $\mathcal{C}$ of $S$. This is an injective map by \textnormal{(LSA1)}. All rings in the right hand side are domains of dimension $N$ by Proposition \ref{proposition_discreteTotallyOrderedToric}. In particular $A$ has the same dimension $N$ and the ideals $I_{\mathcal{C}}/I$ are all the minimal primes of $A$; indeed if $A$ had dimension bigger than $N$ or there were other minimal primes then the above map could not be injective. This proves that the subvarieties $X_{\mathcal{C}}$ are the irreducible components of $X$.

Finally, by Proposition \ref{proposition_discreteTotallyOrderedToric}, a subvariety $X_\calC$, with $\calC$ a chain in $S$, is an irreducible projective (normal) toric variety.

We summarize this discussion in the following
\begin{theorem}
The subvarieties $X_\calC$, as $\calC$ runs over the set of maximal chains of $S$, are the irreducible components of $X$. They are projective (normal) toric varieties and intersect along (normal) toric subvarieties.
\end{theorem}

\noindent This answers some questions posed by Knutson in \cite{knutson} (see Section~5 there).

\subsection{Group Quotient} In the following theorem we see that the discrete LS algebra over a totally ordered set is the ring of invariant of a certain finite abelian group.
\begin{theorem}\label{theorem_quotient}
Suppose that the base field is $\C$. The discrete LS algebra over a totally ordered set of length $N$ with bonds is the ring of invariants for the action of a finite abelian group $G\subseteq\GL_{N+1}(\C)$ on a polynomial ring. Furthermore $G$ can be chosen to contain no pseudo-reflection.
\end{theorem}
\begin{proof}
Let $S=\{\sigma_0<\sigma_1<\cdots<\sigma_N\}$, $b_i=\bond(\sigma_i,\sigma_{i+1})$ for $i=0,\ldots,N-1$, set also $b_{-1}=b_N=1$, $M_i=\lcm(b_{i-1},b_i)$ for $i = 0,1,\ldots,N$, and finally $M=\lcm(b_0,b_1,\ldots,b_{N-1})$. Denote by $\zeta$ a $M$--th primitive root of unity in $\C$ and, for $i = 1,\ldots,N$, define the diagonal matrix
\[
e_i = (\zeta^{\frac{M}{M_0}b_i},\zeta^{\frac{M}{M_1}b_i},\ldots,\zeta^{\frac{M}{M_i}b_i},1,\ldots,1)\in\GL_{N+1}(\C).
\]
The finite abelian group $G$ generated by $e_1,e_2,\ldots,e_N$ acts linearly on the polynomial ring $R=\C[x_{\sigma_0},\ldots,x_{\sigma_N}]$. A generic monomial $\mathsf{m} = \prod_{j=0}^Nx_{\sigma_j}^{n_j}$ is an eigen-vector for $e_i$ with eigenvalue
\[
\zeta^{M b_i\sum_{0\leq j\leq i}\frac{n_j}{M_j}}.
\]
It is then clear that $R^G$ is spanned by the monomials $\mathsf{m}$ such that $b_i\sum_{0\leq j\leq i}\frac{n_j}{M_j}$ is an integer for all $i=0,1,\ldots,N$, i.e. by the monomials $x^\pi$ with $\pi$ an LS-paths. So $R^G$ is the discrete algebra over $S$.

Now we show that $G$ contains no pseudo-reflection. Indeed consider a generic element $e_1^{t_1}\cdots e_N^{t_N}=(\zeta^{a_0},\zeta^{a_1},\ldots,\zeta^{a_N})$ of $G$, where we have set
\[
a_i = \displaystyle\frac{M}{M_i}\sum_{i\leq j\leq N}b_jt_j
\]
with $t_0=0$. Note that the following four statements are equivalent:
\begin{equation}
\zeta^{a_i}=1\\
\end{equation}
\begin{equation}
\frac{M}{M_i}\sum_{i\leq j\leq N}b_jt_j \equiv 0 \pmod{M}\\
\end{equation}
\begin{equation}
\sum_{i\leq j\leq N}b_jt_j \equiv 0 \pmod{M_i}\\
\end{equation}
\begin{equation}
\sum_{i\leq j\leq N}b_jt_j \equiv 0 \pmod{b_{i-1}}\quad\textrm{and}\quad\sum_{i\leq j\leq N}b_jt_j \equiv 0 \pmod{b_i}.\\
\end{equation}
Hence if we suppose $\zeta^{a_i}=1$ for all $0\leq i\leq N$ but a single $j$ then we conclude also $\zeta^{a_j}=1$. This proves that $G$ contains no pseudo-reflection.
\end{proof}
So if $S$ is totally ordered, $\Proj(A)$ is the quotient $\Pj_{\C}^N / G$ for $G\subseteq\GL_{N+1}(\C)$ as in the proof of the proposition. Hence $\Proj(A)$ has some similarities with a weighted projective space. It turns out however that, in general, the toric varieties associated to discrete LS algebras over totally ordered sets are \emph{not} weighted projective spaces (see \cite{CFL} for an example).

Note that this theorem gives another proof that an LS algebra over a totally ordered poset is Cohen-Macaulay. The absence of pseudo-reflections will be used in a criterion for the Gorenstein property in the next section.

\subsection{The Gorenstein Property} As recalled above, the Cohen-Macaulay property of an LS algebra is a property of the order complex $\Delta(S)$. By contrast the Gorenstein property really depends on the bonds; for example $\K[x^2,xy,y^2]$, the discrete algebra over $({\sigma_0<\sigma_1,\bond(\sigma_0,\sigma_1)=2})$, is Gorenstein whereas $\K[x^3,x^2y,xy^2,y^3]$, the discrete algebra over $({\sigma_0<\sigma_1,\bond(\sigma_0,\sigma_1)=3})$, is not Gorenstein.

Following Stanley work \cite{stanley} we study the Gorenstein property via the Hilbert series. If $A=A_0\oplus A_1\oplus A_2\oplus\cdots$ is a graded $\K$--algebra of dimension $n$ with $d_r=\dim_\K A_r$ for all $r\geq 0$, let
\[
F_A(t) = \displaystyle\sum_{r\geq 0}d_rt^r
\]
be its Hilbert series. Here is a criterion by Stanley \cite{stanley}.
\begin{theorem}
\begin{itemize}
\item[(i)] If $A$ is Gorenstein then
\[
F_A(1/t) = (-1)^n t^k F_A(t)
\]
for some integer $k$.
\item[(ii)] If $A$ is a Cohen-Macaulay integral domain, then it is Gorenstein if and only if
\[
F_A(1/t) = (-1)^n t^k F_A(t)
\]
for some integer $k$.
\end{itemize}
\end{theorem}
Now for an LS algebra $A$ we have $d_r = |\LS_r|$, in particular $A$ has the same Hilbert series of the discrete LS algebra over $(S,\leq,\bond)$. Hence, by Theorem \ref{theorem_degeneration}, we find at once
\begin{corollary}\label{corollary_gorenstein}
Let $A$ be an LS algebra over $(S,\leq,\bond)$ that is a a Cohen-Macaulay domain. If the discrete LS algebra over $(S,\leq,\bond)$ is Gorenstein then $A$ is Gorenstein.
\end{corollary}

For a totally ordered set we have the following neat criterion. Let $S=\{\sigma_0<\sigma_1<\cdots<\sigma_N\}$, $b_i=\bond(\sigma_i,\sigma_{i+1})$ for $i=0,\ldots,N-1$, set also $b_{-1}=b_N=1$, $M_i=\lcm(b_{i-1},b_i)$ for $i=0,1,\ldots,N$
\begin{theorem}
Suppose that the base field is $\C$. An LS algebra over $(S,\leq,\bond)$ with $S$ totally ordered is Gorenstein if and only if
\[
S\ni\sigma_i\,\longmapsto\,\displaystyle\frac{1}{M_0}+\frac{1}{M_1}+\cdots+\frac{1}{M_i}\in\Q
\]
is an LS-path.
\end{theorem}
\begin{proof} Let $A$ be an LS algebra over $(S,\leq,\bond)$. Then $\Delta(S)$ is Cohen-Macaulay being $S$ totally ordered; further $A$ is a domain by Remark \ref{remark_discreteIsDomain}. So it suffices to prove that the discrete algebra $A_0$ over $(S,\leq,\bond)$ is Gorenstein by Corollary \ref{corollary_gorenstein}.

We know by Theorem \ref{theorem_quotient} that $A_0$ is the fixed subring of $\C[x_{\sigma_0},\ldots,x_{\sigma_N}]$ by the action of a finite abelian group $G\subseteq\GL_{N+1}(\C)$; let us take for $G$ the group defined in the proof of Theorem 6.1 and keep using the notation there, in particular $G$ contains no pseudo-reflection. Then by \cite{stanley} $A_0$ is Gorenstein if and only if $G$ is contained in $\SL_{N+1}(\C)$. Now the determinant of the element $e_i$ is $\zeta^a$ with
\[
a = Mb_j\displaystyle\sum_{0\leq j\leq i}\frac{1}{M_j}.
\]
Hence this determinant is $1$ if and only if
\[
b_j\displaystyle\sum_{0\leq j\leq i}\frac{1}{M_j}\in\Z.
\]
So $G\subseteq\SL_{N+1}(\C)$ if and only if the map
\[
S\ni\sigma_i\,\longmapsto\,\displaystyle\frac{1}{M_0}+\frac{1}{M_1}+\cdots+\frac{1}{M_i}\in\Q
\]
is an LS-path.
\end{proof}

As a final remark, note that the necessary condition
\[
\displaystyle\frac{1}{M_0}+\frac{1}{M_1}+\cdots+\frac{1}{M_N}\in\N
\]
implies that there exist only a finite number of Gorenstein discrete LS algebras over a totally ordered set of a fixed dimension $N$.

\subsection{The Koszul Property} Let $A=A_0\oplus A_1\oplus\cdots$ be a graded $\K$--algebra finitely generated in degree $1$ with $A_0=\K$; denote by $A_+$ the ideal $A_1\oplus A_2\oplus\cdots$ and consider $\K=A/A_+$ as a graded $A$--module. The algebra $A$ is said to be \emph{Koszul} (or \emph{wonderful}) if $\Tor_i^A(\K,\K)$ is concentrated in degree $i$ for any $i\geq0$. (See \cite{kempf} for this definition.)

In this section we study the Koszul property for LS algebras. We begin with the following result about Gr\"obner basis; its proof is a direct generalization of the same result for the homogeneous coordinate ring of a Schubert variety in \cite{LLM}.
\begin{proposition}
The straightening relations in \textnormal{(LSA2)} gives a quadratic Gr\"obner basis for the defining ideal of an LS algebra as a quotient of the polynomial ring $\K[x_\pi\,|\,\pi\in\LS_1]$ via the map $x_\pi\longmapsto\pi$.
\end{proposition}
\begin{proof} Let $A$ be an LS algebra over the poset with bonds $(S,\leq,\bond)$, let $R = \K[x_\pi\,|\,\pi\in\LS_1]$ and let $I$ be the ideal of $R$ generated by
\[
\mathcal{R}_{\pi_1,\pi_2} = x_{\pi_1}x_{\pi_2}-\displaystyle\sum_j c_jx_{\pi_{j,1}}x_{\pi_{j,2}}
\]
for any non standard monomial $\pi_1\pi_2$ with straightening relation as in \textnormal{(LSA2)}
\[
\pi_1\pi_2=\sum_j c_j\pi_{j,1}\pi_{j,2}.
\]
By Remark \ref{remark_straighteningRelations}, $R/I$ is isomorphic to $A$.

Now fix a total order $\leqt$ refining the given order $\leq$ of $S$ and consider the associated reverse lexicographic order $\leqtrlex$ on LS-paths. Let $\cleq$ be the following order on monomials in $R$
\[
x_{\pi_1}x_{\pi_2}\cdots x_{\pi_r}\,\cleq\,x_{\pi_1'}x_{\pi_2'}\cdots x_{\pi'_{s}}
\]
if $r<s$ or $r=s$ and $\pi_1 \pfs \cdots \pfs \pi_r \leqtrlex \pi'_1 \pfs \cdots \pfs \pi'_r$. This is clearly a monomial order for $R$; moreover $\mathsf{in}(\mathcal{R}_{\pi_1,\pi_2})=x_{\pi_1}x_{\pi_2}$ by the order requirement in (LSA2). Hence the ideal $I'$ spanned by the initial terms of the $\mathcal{R}_{\pi_1,\pi_2}$'s is exactly the $\K$--vector space spanned by the non-standard monomials.

We want to prove that $\mathsf{in}(I)=I'$, i.e. that if $f\in I$ then its initial term $f_0$ cannot be a standard monomial. Suppose otherwise and let $f=f_0+f_1$ with $f_1$ a sum of strictly greater (w.r.t $\cleq$) monomials and $f_0$ standard. Then, using (LSA2), we replace any non-standard monomial $\mathsf{m}$ in $f_1$ by a sum of standard monomials $\mathsf{m}'$ with $\mathsf{m}\cleq\mathsf{m}'$; note that this can't delete $f_0$. By repeating this argument, in a finite number of steps we find a new element $f'\in I$ that is a linear combination of standard monomials. But this is impossible since in $A$ the standard monomials are linearly independent by (LSA1).

This finishes the proof that the $\mathcal{R}_{\pi_1,\pi_2}$'s, with $\pi_1\pi_2$ non-standard, are a Gr\"obner basis for $I$.
\end{proof}
The flat degeneration associated to the Gr\"obner basis in the previous proposition is more coarse than the degeneration to the discrete LS algebra stated in Theorem \ref{theorem_degeneration}. However it can be used to prove the following
\begin{theorem}
An LS algebra is Koszul.
\end{theorem}
\begin{proof} Let $A$ be an LS algebra over the poset with bonds $(S,\leq,\bond)$. By the previous proposition $A$ is the quotient of a polynomial ring $R$ by a homogeneous ideal $I$ admitting a quadratic Gr\"obner basis. In particular $A/\mathsf{in}(I)$ is Koszul by \cite{froberg} being $\mathsf{in}(I)$ a quadratic monomial ideal. But, by Theorem \ref{theorem_degeneration}, there exists a flat $\K[t]$--algebra $\mathcal{A}$ such that $\mathcal{A}/(t-a)$ is isomorphic to $A$ for all $a\in\K^*$ while $\mathcal{A}/(t)$ is isomorphic to $A/\mathsf{in}(I)$; then also $A$ is Koszul by \cite{kempf}.
\end{proof}

\end{document}